\documentclass[12pt]{article}
\usepackage{amsmath,amsthm,amssymb, color}
\input amssym.def
\input amssym.tex
\date{}

\newtheorem{theorem}{THEOREM}

\newtheorem{remark}{Remark}
\newtheorem{problem}{Problem}

\begin{document}

\title{{\bf $L_p +L_{\infty}$ and $L_p \cap L_{\infty}$ are not isomorphic  \newline for all $1 \leq p < \infty$, $p \neq 2$}}
\author {Sergei V. Astashkin\thanks{Research partially supported by
the Ministry of Education and Science of the Russian Federation.} \, and Lech Maligranda}

\date{}

\maketitle

\renewcommand{\thefootnote}{\fnsymbol{footnote}}

\footnotetext[0]{2010 {\it Mathematics Subject Classification}:
46E30, 46B20, 46B42}
\footnotetext[0]{{\it Key words and phrases}: symmetric spaces, isomorphic spaces, complemented subspaces}

\vspace{-5mm}

\begin{abstract}
\noindent {We prove the result stated in the title. It comes as a consequence of the fact that the space 
$L_p \cap L_{\infty}$, $1\le p<\infty$, $p\ne 2$, does not contain a complemented subspace 
isomorphic to $L_p$. In particular, as a subproduct, we show that $L_p \cap L_{\infty}$ contains 
a complemented subspace isomorphic to $l_2$ if and only if $p = 2$. }
\end{abstract}

\section{Preliminaries and main result}

Isomorphic classification of symmetric spaces is an important problem related to the study
of symmetric structures in arbitrary Banach spaces. This research was initiated in the seminal work
of Johnson, Maurey, Schechtman and Tzafriri \cite{JMST}. Somewhat later it was extended by Kalton to 
lattice structures \cite{Ka}. 

In particular, in \cite{JMST} (see also \cite[Section 2.f]{LT}) it was 
shown that the space $L_2 \cap L_p$ for $2 \leq p < \infty$ (resp. $L_2+L_p$ for $1 < p \leq 2$) is isomorphic to $L_p$. 
A detailed investigation of various properties of {\it separable} sums and intersections of $L_p$-spaces (i.e., with $p<\infty$) was undertaken by Dilworth
in the papers \cite{Di88} and \cite{Di90}. In contrast to that, we focus here on the problem if the {\it nonseparable} spaces 
$L_p +L_{\infty}$ and $L_p \cap L_{\infty}$, $1\le p<\infty$, are isomorphic or not.
 
In this paper we use the standard notation from the theory of symmetric spaces (cf. \cite{BS}, \cite{KPS} and \cite{LT}). 
For $1 \leq p < \infty$ the space $L_p +L_{\infty}$ consists of all sums of $p$-integrable and bounded 
measurable functions on $(0, \infty)$ with the norm defined by
\begin{equation*}
\| x \|_{L_p + L_{\infty}}: = \inf_{x(t) = u(t) + v(t), u \in L_p, v \in L_{\infty}} \left( \| u \|_{L_p} + \| v \|_{L_{\infty}} \right).
\end{equation*}
The  $L_p \cap L_{\infty}$ consists of all bounded $p$-integrable functions on $(0, \infty)$ with the norm
\begin{equation*}
\| x \|_{L_p \cap L_{\infty}}: = \max \left \{ \| x \|_{L_p}, \| x \|_{L_{\infty}} \right \} =  \max \left \{ \Big(\int_0^{\infty} | x(t)|^p\, dt\Big)^{1/p}, 
\underset{t > 0}{\rm ess \,sup} | x(t)| \right \}.
\end{equation*}
Both $L_p +L_{\infty}$ and $L_p \cap L_{\infty}$ for all $1 \leq p < \infty$ are non-separable Banach spaces 
(cf. \cite[p. 79]{KPS} for $p = 1$).
The norm in $L_p +L_{\infty}$ satisfies the following sharp estimates
\begin{equation} \label{1}
\Big(\int_0^1 x^*(t)^p \, dt\Big)^{1/p} \leq \| x \|_{L_p + L_{\infty}} \leq 2^{1-1/p} \Big(\int_0^1 x^*(t)^p \, dt\Big)^{1/p} 
\end{equation}
(cf. \cite[p. 109]{BL}, \cite[p. 176]{Ma84} and with details in \cite[Theorem 1]{Ma13}) -- see also \cite[pp. 74-75]{BS} 
and \cite[p. 78]{KPS}, where we can find a proof of (\ref{1}) in the case when $p = 1$, that is, 
$$
\| x \|_{L_1 + L_{\infty}} = \int_0^1 x^*(t) \, dt.
$$
Here, $x^*(t)$ denotes the decreasing rearrangement of $| x(u)|$, that is, 
\begin{equation*}
 x^*(t) = \inf \{\tau>0 \colon m(\{u > 0 \colon |x(u)| > \tau\}) < t \} 
\end{equation*}
(if $E\subset\mathbb{R}$ is a measurable set, then $m(E)$ is its Lebesgue measure). Note that every 
measurable function and its decreasing rearrangement are equimeasurable, that is,
$$
m(\{u > 0 \colon |x(u)| > \tau\})=m(\{t > 0 \colon |x^*(t)| > \tau\})$$
for all $\tau>0$.

Denote by $L_{\infty}^0$ and $(L_p +L_{\infty})^0, 1 \leq p < \infty$, the closure of $L_1 \cap L_{\infty}$ in $L_{\infty}$ 
and in $L_p +L_{\infty}$, respectively. Clearly, $(L_p +L_{\infty})^0 = L_p+L_{\infty}^0$. Note that 
\begin{equation} \label{2}
L_p + L_{\infty}^0 = \{x \in L_p +L_{\infty}\colon x^*(t) \rightarrow 0 ~ {\rm as} ~ t \rightarrow \infty\}
\end{equation}
and 
\begin{equation*}
(L_1 + L_{\infty}^0)^* =  L_1 \cap L_{\infty},
\end{equation*}
i.e., $L_1 \cap L_{\infty}, 1 \leq p < \infty$, is a dual space (cf. \cite[pp. 79-80]{KPS} and \cite[pp. 76-77]{BS}). 
Also, $L_p \cap L_{\infty}$ and $L_p + L_{\infty}, 1 < p < \infty$, are dual spaces because
\begin{equation*}
(L_q + L_1)^* = L_p \cap L_{\infty} ~~{\rm and} ~~(L_q \cap L_1)^* = L_p + L_{\infty}, 
\end{equation*}
where $1/p +1 /q = 1$.
\vspace{2mm}

Now, we state the main result of this paper.

\begin{theorem} \label{Thm1}
For every $1 \leq p < \infty, p \neq 2$, the spaces $L_p +L_{\infty}$ and $L_p \cap L_{\infty}$ are not isomorphic.
\end{theorem}

Clearly, the space $L_p +L_{\infty}$ contains the complemented subspace $(L_p +L_{\infty})_{{\big |} [0, 1]}$ isomorphic to $L_p[0, 1]$
for every $1 \leq p < \infty$. As a bounded projection we can take the operator $Px: = x \chi_{[0, 1]}$ because 
$$
\| Px \|_{L_p} = \| x \chi_{[0, 1]} \|_{L_p} = \Big(\int_0^1 |x(t)|^p \, dt\Big)^{1/p} \leq  \Big(\int_0^1 x^*(t)^p \, dt\Big)^{1/p} \leq \| x \|_{L_p +L_{\infty}}.
$$
In the next two sections we show that $L_p \cap L_{\infty}$ for $p \in [1, 2) \cup (2, \infty)$ does not contain a complemented subspace 
isomorphic to $L_p$, which gives our claim.
At the same time, note that $L_p \cap L_{\infty}$, $1\le p<\infty$, contains a subspace isomorphic to $L_{\infty}$ and hence a subspace  
isomorphic to $L_p$.

The spaces $L_p +L_{\infty}$ and $L_{\infty}$ are not isomorphic since $L_p +L_{\infty}$ contains a complemented subspace 
isomorphic to $L_{p}$ and $L_{\infty}$ is a prime space (this follows from the Lindenstrauss and Pe{\l}czy\'nski results -- see 
\cite[Theorems 5.6.5 and 4.3.10]{AK}). Similarly, the spaces $L_p \cap L_{\infty}$ and $L_{\infty}$ are not isomorphic because 
of $L_p \cap L_{\infty}$ contains a complemented subspace isomorphic to $l_p$ (take, for instance, the span of the sequence 
$\{\chi_{[n-1, n)} \}_{n=1}^{\infty}$ in $L_p \cap L_{\infty}$).

If $\{x_n\}_{n=1}^\infty$ is a sequence from a Banach space $X$, by $[x_n]$ we denote its closed linear span in $X$. As usual,
the Rademacher functions on $[0, 1]$ are defined as follows: $r_k(t) = {\rm sign}
(\sin 2^k \pi t), ~ k \in {\Bbb N}, t \in [0, 1]$.

\section{ $L_1 \cap L_{\infty}$ does not contain a complemented subspace isomorphic to $L_1$}

Our proof of Theorem \ref{Thm1} in the case $p = 1$ will be based on an application of the Hagler-Stegall theorem proved in \cite{HS} 
(see Theorem 1). To state it we need the following definition.

The space $(\bigoplus_{n=1}^{\infty} l_\infty^n)_{l_p}$, $1 \leq p < \infty$, is the Banach space of all sequences $\{c^n_k\}_{n=1}^{\infty}$,   
$(c^n_k)_{k=1}^n \in l_\infty^n$, $n=1,2,\dots,$ such that
$$
\| \{c^n_k \} \| := \Big(\sum_{n = 1}^{\infty} \| (c^n_k)_{k=1}^n \|_{l_\infty}^p\Big)^{1/p} = \Big(\sum_{n = 1}^{\infty} \max_{1 \leq k \leq n} |c^n_k |^p\Big)^{1/p} < \infty.
$$

\begin{theorem}[Hagler-Stegall] \label{Thm2} Let $X$ be a Banach space. Then its dual $X^*$ contains a complemented subspace 
isomorphic to $L_1$ if and only if $X$ contains a subspace isomorphic to $(\bigoplus_{n=1}^{\infty} l_\infty^n)_{l_1}$.
\end{theorem}

Note that $(L_1 + L_{\infty}^0)_{{\big |} [0, 1]} = L_1[0, 1]$, and hence $L_1 + L_{\infty}^0$ contains a complemented copy of $L_1[0, 1]$, 
and so of $l_1$. Moreover, its subspace 
\begin{equation}
\left\{ \sum_{k=1}^{\infty} c_k \chi_{[k-1, k]}\colon c_k \rightarrow 0 ~ {\rm as} ~ k \rightarrow \infty \right\}
\end{equation}
is isomorphic to $c_0$ and so, by Sobczyk theorem (cf. \cite[Theorem 2.5.8]{AK}), is complemented in the separable space $L_1 + L_{\infty}^0$.  
Therefore, the latter space contains uniformly complemented copies of $l_{\infty}^n, n \in \mathbb N$. However, we have
 
\begin{theorem} \label{Thm3} The space $L_1 + L_{\infty}^0$ does not contain any subspace isomorphic to the space 
$(\bigoplus_{n=1}^{\infty} l_\infty^n)_{l_1}$.
\end{theorem}
\begin{proof} On the contrary, assume that $L_1 + L_{\infty}^0$ contains a subspace isomorphic to $(\bigoplus_{n=1}^{\infty} l_\infty^n)_{l_1}$. 
Let $x_k^n, n \in {\mathbb N}, k = 1, 2, \ldots, n$, form the sequence from $L_1 + L_{\infty}^0$ equivalent to the unit vector basis of 
$(\bigoplus_{n=1}^{\infty} l_\infty^n)_{l_1}$. This means that there is a constant $C > 0$ such that for all $a_k^n \in \mathbb R$ 
\begin{equation*}
C^{-1} \sum_{n=1}^{\infty} \max_{k = 1, 2, \ldots, n} |a_k^n| \leq \Big\|  \sum_{n=1}^{\infty}  \sum_{k=1}^n a_k^n x_k^n \Big\|_{L_1 +L_{\infty}} 
\leq C \sum_{n=1}^{\infty} \max_{k = 1, 2, \ldots, n} |a_k^n|.
\end{equation*}
In particular, for any $n \in \mathbb N$, every subset $A \subset \{1, 2, \ldots, n\}$ and all $\varepsilon_k = \pm 1, k \in A$, we have
\begin{equation} \label{4}
\Big\| \sum_{k \in A} \varepsilon_k x_k^n \Big\|_{L_1 +L_{\infty}} = \int_0^1 \Big(\sum_{k \in A} \varepsilon_k x_k^n\Big)^*(s) \, ds \leq C,
\end{equation}
and for all $1 \leq k(n) \leq n, n \in \mathbb N$ the sequence $\{x_{k(n)}^n \}_{n=1}^{\infty}$ is equivalent in $L_1+L_\infty$ to the unit vector basis of $l_1$, i.e., 
for all $a_n \in \mathbb R$
\begin{equation} \label{5}
C^{-1} \sum_{n=1}^{\infty} |a_n| \leq \int_0^1 \Big(\sum_{n=1}^{\infty} a_n x_{k(n)}^n\Big)^*(s) \, ds \leq C \sum_{n=1}^{\infty} |a_n|.
\end{equation}
Moreover, we can assume that $\| x_k^n \|_{L_1 +L_{\infty}} = 1$ for all $n \in {\mathbb N}, k = 1, 2, \ldots, n$, i.e.,
\begin{equation} \label{6}
\int_0^1 (x_k^n )^*(s)\, ds = 1, n \in {\mathbb N}, k = 1, 2,\ldots, n.
\end{equation}

Firstly, we show that for every $\delta > 0$ there is $M = M(\delta) \in \mathbb N$ such that for all $n \in \mathbb N$ and any 
$E \subset (0, \infty)$ with $m(E) \leq 1$ we have
\begin{equation} \label{7}
{\rm card} \{k = 1, 2, \ldots, n \colon \int_E |x_k^n(s)| ds \geq \delta \} \leq M.
\end{equation}
Indeed, assuming the contrary, for some $\delta_0 > 0$ we can find $n_i \uparrow, E_i \subset (0, \infty), m(E_i) \leq 1, i = 1, 2, \ldots$, 
such that
\begin{equation*}
{\rm card} \{k = 1, 2, \ldots, n_i \colon \int_{E_i} |x_k^{n_i}(s)| ds \geq \delta_0 \}  \rightarrow \infty .
\end{equation*}
Denoting $A_i:= \{k = 1, 2, \ldots, n_i \colon  \int_{E_i} |x_k^{n_i}(s)| ds \geq \delta_0 \}$, for all $\varepsilon_k = \pm 1$ we have
\begin{equation} \label{8}
\Big\| \sum_{k \in A_i} \varepsilon_k x_k^{n_i} \Big\|_{L_1 +L_{\infty}} = \int_0^1 \Big(\sum_{k \in A_i} \varepsilon_k x_k^{n_i}\Big)^*(s) \, ds 
\geq  \int_{E_i} \Big|\sum_{k \in A_i} \varepsilon_k x_k^{n_i}(s) \Big|\, ds.
\end{equation}
Moreover, by the Fubini theorem, Khintchine's inequality in $L_1$ (cf. \cite[pp. 50-51]{LT} or \cite{Sz}) and Minkowski inequality, 
we obtain
\begin{eqnarray*}
\int_0^1 \int_{E_i} \Big| \sum_{k \in A_i} r_k(t) x_k^{n_i}(s) \Big|\, ds\, dt 
&=&
\int_{E_i} \int_0^1 \Big| \sum_{k \in A_i} r_k(t) x_k^{n_i}(s) \Big|\, dt\, ds \\
&\geq& 
\frac{1}{\sqrt{2}} \int_{E_i} \Big(\sum_{k \in A_i} |x_k^{n_i}(s) |^2\Big)^{1/2} \, ds \\
&\geq& 
\frac{1}{\sqrt{2}} \left( \sum_{k \in A_i} \Big( \int_{E_i} |x_k^{n_i}(s) | \, ds\Big)^2 \right)^{1/2} 
\geq \frac{\delta_0}{\sqrt{2}} \sqrt{{\rm card} A_i}.
\end{eqnarray*}
Therefore, for each $i \in \mathbb N$ there are signs $\varepsilon_k(i), k \in A_i$ such that
\begin{equation*}
 \int_{E_i} \Big| \sum_{k \in A_i} \varepsilon_k(i) x_k^{n_i}(s) \Big|\, ds \geq  \frac{\delta_0}{\sqrt{2}} \sqrt{{\rm card} A_i}.
\end{equation*}
Combining this with (\ref{8}) we obtain that
\begin{equation*} 
\Big\| \sum_{k \in A_i} \varepsilon_k(i) x_k^{n_i} \Big\|_{L_1 +L_{\infty}}  \geq  \frac{\delta_0}{\sqrt{2}} \sqrt{{\rm card} A_i}, i = 1, 2, \ldots .
\end{equation*}
Since ${\rm card} A_i \rightarrow \infty$ as $i \rightarrow \infty$, the latter inequality contradicts (\ref{4}). Thus, (\ref{7}) is proved.

Now, we claim that for all $\delta > 0$ and $n \in \mathbb N$ 
\begin{eqnarray} \label{9}
{\rm card} \Big {\{}k = 1, 2, \ldots, n \colon & {\rm there ~is} ~
F \subset [0, \infty) ~{\rm such ~that} ~m(F) \leq \frac{1}{M+1}  \nonumber \\
&~{\rm and} ~ \int_F |x_k^n(s)| \, ds \geq \delta \Big {\}} \leq M,
\end{eqnarray}
where $M$ depending on $\delta$ is taken from (\ref{7}).

Indeed, otherwise, we can find $\delta' > 0, n_0 \in \mathbb N$ and $I \subset \{1, 2, \ldots, n_0\}$, card$I = M_0 + 1, M_0 = M(\delta_0),$
such that for every $k \in I$ there is $F_k \subset (0, \infty)$ with 
$$
m(F_k) \leq \frac{1}{M_0+1} ~~  {\rm and} ~~ \int_{F_k} |x_k^{n_0}(s)| \, ds \geq \delta'.
$$ 
Setting $E = \bigcup_{k \in I} F_k$, we see that $m(E) \leq \sum_{k \in I} m(F_k) \leq 1$.
Moreover, by the definition of $I$ and $E$,  
\begin{equation*} 
{\rm card} \{k = 1, 2, \ldots, n_0 \colon \int_E  |x_k^{n_0}(s)| \, ds \geq \delta' \} \geq {\rm card} I > M_0,
\end{equation*}
which is impossible because of (\ref{7}).

Now, we construct a special sequence of pairwise disjoint functions, which is equivalent in $L_1 + L_{\infty}^0$ 
to the unit vector basis in $l_1$. By (\ref{7}), for arbitrary $\delta_1 > 0$ there is $M_1 = M_1(\delta_1)  \in \mathbb N$ 
such that for all $n \in \mathbb N$
$$
{\rm card} \{k = 1, 2, \ldots, n \colon \int_0^1  |x_{k_1}^{n_1}(s)| \, ds \geq \delta_1 \} \leq M_1. 
$$
Therefore, taking $n_1 > 2 M_1$, we can find $k_1 = 1, 2, \ldots, n_1$ satisfying
$$
\int_0^1  |x_{k_1}^{n_1}(s)| \, ds < \delta_1
$$
and, by (\ref{9}), such that from $F \subset (0, \infty)$ with $m(F) \leq \frac{1}{M_1+1}$ it follows that 
$$
\int_F  |x_{k_1}^{n_1}(s)| \, ds < \delta_1.$$

Moreover, recalling (\ref{2}) we have $(x_{k_1}^{n_1})^*(t) \rightarrow 0$ as $t \rightarrow \infty$. Therefore, 
since $x_{k_1}^{n_1}\in L_1+L_\infty$ and any measurable function is equimeasurable with its
decreasing rearrangement, there exists $m_1 \in \mathbb N$ such that 
$\| x_{k_1}^{n_1} \chi_{[m_1, \infty)} \|_{L_1 + L_{\infty}} \leq \delta_1$. Then, setting $y_1:= x_{k_1}^{n_1} \chi_{[1, m_1]}$, 
we have
\begin{equation*} 
\| x_{k_1}^{n_1} - y_1 \|_{L_1 + L_{\infty}} \leq 2 \delta_1.
\end{equation*}
Next, by (\ref{7}), for arbitrary $\delta_2 > 0$ there is $M_2 = M_2(\delta_2) \in \mathbb N$ such that for all $n \in \mathbb N$ and 
$j = 1, 2, \ldots, m_1$ 
$$
{\rm card} \{k = 1, 2, \ldots, n \colon \int_{j - 1}^j  |x_k^n(s)| \, ds \geq \delta_2 \} \leq M_2. 
$$
Let $n_2 \in \mathbb N$ be such that $n_2 > M_2 m_1 + M_2 + M_1$. Then, by the preceding inequality and (\ref{9}), there is 
$1 \leq k_2 \leq n_2$ such that for all $j = 1, 2, \ldots, m_1$ we have 
\begin{eqnarray} \label{e1}
 \int_{j - 1}^j  |x_{k_2}^{n_2}(s)| \, ds \leq \delta_2, 
\end{eqnarray}
and from $F \subset (0, \infty)$ with $m(F) \leq \frac{1}{M_i+1}, i = 1, 2$, it follows that
\begin{equation*}
 \int_F |x_{k_2}^{n_2}(s)| \, ds \leq \delta_i.
\end{equation*}
Note that \eqref{e1} implies $\int_0^{m_1}  |x_{k_2}^{n_2}(s)| \, ds \leq m_1 \delta_2$, whence 
\begin{equation*} 
\| x_{k_2}^{n_2} \chi_{[0, m_1]} \|_{L_1 + L_{\infty}} \leq m_1 \delta_2.
\end{equation*}

As above, by (\ref{2}), there is $m_2 > m_1$ such that $\| x_{k_2}^{n_2} \chi_{[m_2, \infty)} \|_{L_1 + L_{\infty}} \leq m_1 \delta_2$. 
Thus, putting $y_2:= x_{k_2}^{n_2} \chi_{[m_1, m_2]}$, we have
\begin{equation*} 
\| x_{k_2}^{n_2} - y_2 \|_{L_1 + L_{\infty}} \leq 2 m_1 \delta_2.
\end{equation*}
Continuing this process, for any $\delta_3 > 0$, by (\ref{7}), we can find $M_3 \in \mathbb N$ such that for all $n \in \mathbb N$ and 
$j = 1, 2, \ldots, m_2$ it holds
$$
{\rm card} \{k = 1, 2, \ldots, n \colon \int_{j - 1}^j  |x_k^n(s)| \, ds \geq \delta_3 \} \leq M_3. 
$$
So, again, applying (\ref{9}) and taking $n_3 > m_2M_3 + M_1 + M_2 + M_3$ we find $1 \leq k_3 \leq n_3$ such that 
$$ \int_{j - 1}^j  |x_{k_3}^{n_3}(s)| \, ds \leq \delta_3,\;\; j = 1, 2, \ldots, m_2,$$ and
\begin{equation*} 
\int_F |x_{k_3}^{n_3}(s)| \, ds \leq \delta_i, 
\end{equation*}
whenever $m(F) \leq \frac{1}{M_i+1}, i = 1, 2, 3$. This implies that $\int_0^{m_2} |x_{k_3}^{n_3}(s)| \, ds \leq m_2 \delta_3$, and so
\begin{equation*} 
\| x_{k_3}^{n_3} \chi_{[0, m_2]} \|_{L_1 + L_{\infty}} \leq m_2 \delta_3.
\end{equation*}
Choosing $m_3 > m_2$ so that $\| x_{k_3}^{n_3} \chi_{[m_3, \infty)} \|_{L_1 + L_{\infty}} \leq m_2 \delta_3$ and setting 
$y_3:=  x_{k_3}^{n_3} \chi_{[m_2, m_3]}$, we obtain
\begin{equation*} 
\| x_{k_3}^{n_3} - y_3 \|_{L_1 + L_{\infty}} \leq 2 m_2 \delta_3.
\end{equation*}
As a result, we get the increasing sequences $n_i, m_i, k_i$ of natural numbers, $1 \leq k_i \leq n_i, i = 1, 2, \ldots$ and the sequence 
$\{y_i\}$ of pairwise disjoint functions from $L_1 +L_{\infty}^0$ such that
\begin{equation*} 
\| x_{k_i}^{n_i} - y_i \|_{L_1 + L_{\infty}} \leq 2 m_{i-1} \delta_i,
\end{equation*}
where $m_0:= 1$. Noting that the sequence of positive reals $\{ \delta_i\}_{i=1}^{\infty}$ can be chosen in such a way that the numbers 
$ m_{i - 1} \delta_i$ would be arbitrarily small, we can assume, by the principle of small perturbations (cf. \cite[Theorem~1.3.10]{AK})
and by inequalities (\ref{5}), 
that $\{y_i\}$ is equivalent in $L_1 +L_{\infty}$ to the unit vector basis of $l_1$.
Moreover, by construction, for all $j = 1, 2, \ldots$ and $i = 1, 2, 3, \ldots, j$ we have
\begin{equation} \label{10}
\int_F |y_j(s)| \, ds \leq \delta_i ~ {\rm whenever} ~ m(F) \leq \frac{1}{M_i + 1}.
\end{equation}

Let $1 \leq l < m$ be arbitrary. Since $y_i, i = 1, 2, \ldots$ are disjoint functions, then
\begin{equation}  \label{11}
\Big\| \sum_{i = l}^m y_i \Big\|_{L_1 + L_{\infty}} = \int_0^1 \Big(\sum_{i = l}^m y_i\Big)^*(s)\, ds = \sum_{i = l}^m \int_{E_i} |y_i(s)| \, ds,
\end{equation}
where $E_i$ are disjoint sets from $(0, \infty)$ such that $\sum_{i = l}^m m(E_i) \leq 1$. 
Clearly, for a fixed $l$ we have
$$
k_0(m):={\rm card} \{i\in\mathbb{N}: l\le i\le m\;\mbox{and}\; m(E_i) > \frac{1}{M_l + 1} \} \leq M_l + 1. 
$$
Hence, by (\ref{10}), (\ref{11}) and (\ref{6}), 
\begin{equation*} 
\Big\| \sum_{i = l}^m y_i \Big\|_{L_1 + L_{\infty}} \leq k_0(m) + (m - l - k_0(m)) \delta_l.
\end{equation*}
So, assuming that $m\ge (M_l+1)/\delta_l+l$, we obtain
\begin{equation*} 
\Big\| \sum_{i = l}^m y_i \Big\|_{L_1 + L_{\infty}} \leq 2 \delta_l (m-l).
\end{equation*}
Since $\delta_l \rightarrow 0$ as $l \rightarrow \infty$, the latter inequality contradicts the fact that $\{y_i\}$ is equivalent in $L_1+L_\infty$
to the unit vector basis of $l_1$. The proof is complete.
\end{proof}

\begin{remark} \label{Rem1} Since the space $L_p$, $1 \leq p < \infty$, is of Rademacher cotype $\max (p, 2)$,  the result of Theorem \ref{Thm3} can be generalized as follows: For every $1 \leq p < \infty$ the space $L_p +L_{\infty}^0$ does not contain any subspace isomorphic to the space $(\bigoplus_{n=1}^{\infty} l_\infty^n)_{l_p}$.
\end{remark}

\begin{proof}[Proof of Theorem \ref{Thm1} for $p = 1$]
By the Hagler-Stegall theorem \ref{Thm2}, Theorem \ref{Thm3} and the fact that $L_1 \cap L_{\infty} = (L_1 + L_{\infty}^0)^*$, we obtain 
that (in contrast to $L_1+L_\infty$) the space $L_1 \cap L_{\infty}$ does not contain a complemented subspace isomorphic to $L_1[0, 1]$, which gives our claim.
\end{proof}

There is a natural question (cf. also \cite[p. 28]{ALM}) if the space $L_1 +L_{\infty}$ is isomorphic to a dual space? Our guess is that
not, but we don't have a proof. Of course, the answer ``not"  would imply immediately the result of Theorem \ref{Thm1} for $p = 1$.

\begin{problem}
Is the space $L_1 +L_{\infty}$ isomorphic to a dual space? 
\end{problem}

\section{ $L_p \cap L_{\infty}$ for $p \neq 2$ does not contain a complemented subspace isomorphic to $L_p$}

The well-known Raynaud's result (cf. \cite[Theorem 4]{Ra}) presents the conditions under which a
separable symmetric space (on $[0,1]$ or on $(0,\infty)$) contains a complemented subspace isomorphic to $l_2$.
The following theorem can be regarded as its extension to a special class of nonseparable spaces.

\begin{theorem} \label{Thm4} Let $1 \leq p < \infty$. Then the space $L_p \cap L_{\infty}$ contains a complemented subspace isomorphic 
to the space $l_2$ if and only if $p = 2$.
\end{theorem}
\begin{proof} If $p = 2$, then clearly the sequence $\{\chi_{[n-1, n)}\}_{n=1}^{\infty}$ is equivalent in $L_2 \cap L_{\infty}$ to the unit vector basis 
of $l_2$ and spans a complemented subspace.

Let us prove necessity. 
On the contrary, let $\{x_n\} \subset L_p \cap L_{\infty}$ be a sequence equivalent in $L_p \cap L_{\infty}$ to the unit vector basis of $l_2$ so that 
$[x_n]$ is a complemented subspace of $L_p \cap L_{\infty}$.

Firstly, let us show that there is not $a > 0$ such that for all $c_k \in {\mathbb R}, k = 1, 2, \ldots$
\begin{equation*}
\Big\| \sum_{k=1}^{\infty} c_k x_k \chi_{[0, a]} \Big\|_{L_1} \asymp \Big\| \sum_{k=1}^{\infty} c_k x_k \Big\|_{L_p \cap L_{\infty}}.
\end{equation*}
Indeed, the latter equivalence implies
\begin{equation*}
\Big\| \sum_{k=1}^{\infty} c_k x_k \chi_{[0, a]} \Big\|_{L_1} \asymp \Big\| \sum_{k=1}^{\infty} c_k x_k \Big\|_{L_p \cap L_{\infty}[0, a]} \asymp \| (c_k) \|_{l_2}.
\end{equation*}
Since $L_p \cap L_{\infty}[0, a] = L_{\infty} [0, a]$, we see that the sequence $\{x_n \chi_{[0, a]} \}$ spans in both spaces $L_1[0, a]$ and 
$L_{\infty}[0, a]$ the same infinite-dimensional space. However, by the well-known Grothendieck's theorem (cf. \cite[Theorem 1]{Gr}; see 
also \cite[p. 117]{Ru91}) it is impossible. As a result, we can find a sequence $\{f_n\} \subset [x_k], \| f_n \|_{L_p \cap L_{\infty}} = 
1, n = 1, 2, \ldots$, such that for every $a > 0$
\begin{equation*}
\int_0^a | f_n(t)| \, dt \rightarrow 0 ~ {\rm as} ~ n \rightarrow \infty.
\end{equation*}
Hence, $f_n \overset{m}{\rightarrow } 0$ (convergence in Lebesgue measure $m$) on any interval $[0, a]$. Since $[x_k]$ spans $l_2$, 
then passing to a subsequence if it is necessary (and keeping the same notation), we can assume that $f_n \rightarrow 0$ weakly in 
$L_p \cap L_{\infty}$. Therefore, combining the Bessaga-Pe{\l}czy\'nski Selection Principle (cf. \cite[Theorem~1.3.10]{AK}) 
and the principle of small perturbations (cf. \cite[Theorem~1.3.10]{AK}), we can select a further subsequence, 
which is equivalent to the sequence $\{x_k\}$ in $L_p \cap L_{\infty}$ (and so to the unit vector basis in $l_2$) and which
spans a complemented subspace in $L_p \cap L_{\infty}$. Let it be denoted still by $\{f_n\}_{n=1}^{\infty}$. 
Now, we will select a special subsequence from $\{f_n\}$, which is equivalent to a sequence of functions whose supports intersect 
only over some subset of $(0, \infty)$ with Lebesgue measure at most $1$.

Let $\{\varepsilon_n\}_{n=1}^{\infty}$ be an arbitrary (by now) decreasing sequence of positive reals, $\varepsilon_1 < 1$. 
Since $f_n \overset{m}{\rightarrow } 0$ on $[0, 1]$, there is $n_1 \in \mathbb N$ such that
\begin{equation} \label{12}
m(\{t \in [0, 1] \colon |f_{n_1} (t)| > \varepsilon_1 \}) <  \varepsilon_1.
\end{equation}
Moreover, the fact that $\| f_{n_1} \chi_{(m, \infty)} \|_{L_p} \rightarrow 0$ as $m \rightarrow \infty$ allows us to find 
$m_1 \in \mathbb N$, for which
\begin{equation} \label{13}
\| f_{n_1} \chi_{[m_1, \infty)} \|_{L_p} \leq \varepsilon_2^2.
\end{equation}
Clearly, from (\ref{13}) it follows that
\begin{equation} \label{14}
m(\{t \in [m_1, \infty) \colon |f_{n_1} (t)| > \varepsilon_2 \}) \leq  \varepsilon_2.
\end{equation}
Denoting
$$
A_1:= \{t \in [0, 1]\colon |f_{n_1}(t)| > \varepsilon_1 \}, ~ B_1^0:= \{t \in [m_1, \infty) \colon  |f_{n_1}(t)| > \varepsilon_2\}
$$
and
$$
g_1:= f_{n_1} \left (\chi_{A_1} + \chi_{B_1^0} + \chi_{[1, m_1]} \right),
$$
from (\ref{12}), (\ref{13}) and (\ref{14}) we have
$$
\| f_{n_1} - g_1 \|_{L_p \cap L_{\infty}} \leq \varepsilon_1 + \max(\varepsilon_2, \varepsilon_2^2) \leq 2\, \varepsilon_1.
$$
Further, since $f_n \overset{m}{\rightarrow } 0$ on $[0, m_1]$, there exists $n_2 > n_1, n_2 \in \mathbb N$ such that
\begin{equation} \label{15}
m(\{t \in [0, m_1] \colon |f_{n_2} (t)| > \frac{\varepsilon_2}{m_1} \}) <  \varepsilon_2.
\end{equation}
Again, using the fact that $\| f_{n_2} \chi_{(m, \infty)} \|_{L_p} \rightarrow 0$ as $m \rightarrow \infty$, we can choose 
$m_2 > m_1$ in such a way that
\begin{equation} \label{16}
\| f_{n_2} \chi_{[m_2, \infty)} \|_{L_p} \leq \varepsilon_3^2.
\end{equation}
and also 
\begin{equation} \label{17}
m(B_1^1) < \varepsilon_3, ~ {\rm where} ~~  B_1^1:= B_1^0 \cap [m_2, \infty).
\end{equation}
From (\ref{16}), obviously, it follows that
\begin{equation} \label{18}
m(\{t \in [m_2, \infty) \colon |f_{n_2} (t)| > \varepsilon_3 \}) \leq  \varepsilon_3.
\end{equation}
Setting
$$
A_2:= \{t \in [0, m_1]\colon |f_{n_2}(t)| > \varepsilon_2 m_1^{-1/p} \}, ~ B_2^0:= \{t \in [m_2, \infty) \colon  
|f_{n_2}(t)| > \varepsilon_3\}
$$
and
$$
g_2:= f_{n_2} \left (\chi_{A_2} + \chi_{B_2^0} + \chi_{[m_1, m_2]} \right),
$$
by (\ref{15}), (\ref{16}) and (\ref{18}), we get
$$
\| f_{n_2} - g_2 \|_{L_p \cap L_{\infty}} \leq \max(\varepsilon_2 m_1^{-1/p}, \varepsilon_2) + \max(\varepsilon_3, \varepsilon_3^2) 
< 2\, \varepsilon_2.
$$
Let's do one more step. Since $f_n \overset{m}{\rightarrow } 0$ on $[0, m_2]$, there exists $n_3 > n_2, n_3 \in \mathbb N$ 
such that
\begin{equation} \label{19}
m(\{t \in [0, m_2] \colon |f_{n_3} (t)| > \varepsilon_3 m_2^{-1/p} \}) < \varepsilon_3.
\end{equation}
As above, we can choose $m_3 > m_2$ with the properties
\begin{equation} \label{20}
\| f_{n_3} \chi_{[m_3, \infty)} \|_{L_p} \leq \varepsilon_4^2,
\end{equation}
\begin{equation} \label{21}
m(B_1^2) < \varepsilon_4, ~ {\rm where} ~~  B_1^2:= B_1^0 \cap [m_3, \infty),
\end{equation}
and
\begin{equation} \label{22}
m(B_2^1) < \varepsilon_4, ~ {\rm where} ~~  B_2^1:= B_2^0 \cap [m_3, \infty).
\end{equation}
From (\ref{20}) we infer that
\begin{equation} \label{23}
m(\{t \in [m_3, \infty) \colon |f_{n_3} (t)| > \varepsilon_4 \}) \leq  \varepsilon_4.
\end{equation}
Finally, putting
$$
A_3:= \{t \in [0, m_2]\colon |f_{n_3}(t)| > \varepsilon_3 m_2^{-1/p} \}, ~ B_3^0:= \{t \in [m_3, \infty) \colon  
|f_{n_3}(t)| > \varepsilon_4\}
$$
and
$$
g_3:= f_{n_3} \left (\chi_{A_3} + \chi_{B_3^0} + \chi_{[m_2, m_3]} \right),
$$
by (\ref{19}), (\ref{20}) and (\ref{23}), we have
$$
\| f_{n_3} - g_3 \|_{L_p \cap L_{\infty}} \leq \max(\varepsilon_3 m_2^{-1/p}, \varepsilon_3) + \max(\varepsilon_4, \varepsilon_4^2) 
< 2\, \varepsilon_3.
$$
Continuing in the same way, we get the increasing sequences of natural numbers $\{n_k\}, \{m_k\}$, the sequences of sets 
$\{A_k\}_{k=1}^{\infty}, \{B_k^i\}_{i=0}^{\infty}, k = 1, 2, \ldots$ and the sequence of functions
$$
g_k := f_{n_k} \left (\chi_{A_k} + \chi_{B_k^0} + \chi_{[m_{k-1}, m_k]} \right),
$$
(where $m_0 = 1$), satisfying the properties
\begin{equation} \label{24}
m(A_k) \leq  \varepsilon_k, k = 1, 2, \ldots,
\end{equation}
\begin{equation} \label{25}
m(B_k^i) \leq  \varepsilon_{k + i + 1}, k = 1, 2, \ldots, i = 0, 1, 2, \ldots, 
\end{equation}
(see (\ref{17}), (\ref{21}) and (\ref{22})) and
$$
\| f_{n_k} - g_k \|_{L_p \cap L_{\infty}} \leq 2\, \varepsilon_k, k = 1, 2, \ldots.
$$
In particular, by the last inequality, choosing sufficiently small $ \varepsilon_k$, $k=1,2,\dots$, and applying once more 
the principle of small perturbations \cite[Theorem~1.3.10]{AK}, we may assume that the sequence $\{g_k\}$ is equivalent 
to $\{f_{n_k}\}$ (and so to the unit vector basis of $l_2$) and the subspace $[g_k]$ is complemented in $L_p \cap L_{\infty}$. 
Thus, for some $C > 0$ and all $(c_k) \in l_2$,  
\begin{equation} \label{26}
C^{-1} \| (c_k) \|_{l_2} \leq \Big\| \sum_{k=1}^{\infty} c_k g_k \Big\|_{L_p \cap L_{\infty}} \leq C \| (c_k) \|_{l_2}.
\end{equation}

Now, denote
$$
C_1:= \bigcup_{i=1}^{\infty} A_i \cup B_1^0, C_2:= \bigcup_{i=2}^{\infty} A_i \cup B_1^0 \cup B_2^0, 
C_3:= \bigcup_{i=3}^{\infty} A_i \cup B_1^1 \cup B_2^0 \cup B_3^0, \ldots
$$
$$
\ldots, C_j:= \bigcup_{i=j}^{\infty} A_i \cup B_1^{j-2} \cup B_2^{j-3} \cup \ldots \cup B_{j-2}^1 \cup B_{j-1}^0 \cup B_j^0, \ldots.
$$
Setting $C:= \bigcup_{j=1}^{\infty} C_j$ and applying (\ref{24}) and (\ref{25}), we have
\begin{equation} \label{27}
m(C) \leq \sum_{j=1}^{\infty} m(C_j) \leq \sum_{j=1}^{\infty} \Big(\sum_{i=j}^{\infty} \varepsilon_i + j \varepsilon_j\Big) \leq 1
\end{equation}
whenever $\varepsilon_k, k = 1, 2, \ldots,$ are sufficiently small. Putting
\begin{eqnarray*}
D_1 
&=& [1, m_1] \setminus \Big(\bigcup_{i = 2}^{\infty} A_i\Big), D_2 = [m_1, m_2] \setminus \Big(\bigcup_{i = 3}^{\infty} A_i \cup B_1^0\Big), \\
D_3 &=& [m_2, m_3] \setminus \Big(\bigcup_{i = 4}^{\infty} A_i \cup B_1^1 \cup B_2^0\Big), \ldots, \\
D_j &=& [m_{j-1}, m_j] \setminus \left( \bigcup_{i = j+1}^{\infty} A_i \cup B_1^{j-2} \cup B_2^{j-3} \cup \ldots \cup B_{j-2}^1 \cup B_{j-1}^0 \right), \ldots,
\end{eqnarray*}
and recalling the definition of $g_k, k = 1, 2, \ldots$, we infer that 
$$
g_k = u_k + v_k, ~ {\rm where} ~~ u_k:= g_k \chi_{C_k} ~ {\rm and} ~~ v_k:= g_k \chi_{D_k}, k = 1, 2, \ldots.
$$
Note that $(\bigcup_{k=1}^{\infty} C_k) \cap (\bigcup_{k=1}^{\infty} D_k) = \varnothing$, whence (\ref{26}) can be rewritten as follows
\begin{equation} \label{28}
\frac{1}{2} \, C^{-1} \| (c_k) \|_{l_2} \leq \max \left(\Big\| \sum_{k=1}^{\infty} c_k u_k \Big\|_{L_p \cap L_{\infty}}, 
\Big\| \sum_{k=1}^{\infty} c_k v_k \Big\|_{L_p \cap L_{\infty}} \right) \leq C \| (c_k) \|_{l_2}.
\end{equation}
Moreover, the subspace $[u_k]$ is also complemented in $L_p \cap L_{\infty}$ and, by (\ref{27}), we have
\begin{equation} \label{29}
m\Big(\bigcup_{k=1}^{\infty} {\rm supp} \, u_k\Big) \leq 1.
\end{equation}

Now, suppose that $\liminf_{k \rightarrow \infty} \| u_k \|_{L_p \cap L_{\infty}} = 0$. Then passing to a subsequence (and keeping the same 
notation), by (\ref{28}), we obtain
\begin{equation} \label{30}
\frac{1}{2 C}\,  \| (c_k) \|_{l_2} \leq \Big\| \sum_{k=1}^{\infty} c_k v_k \Big\|_{L_p \cap L_{\infty}} \leq C \| (c_k) \|_{l_2}.
\end{equation}
Since $v_k, k = 1, 2, \ldots$, are pairwise disjoint, we have
\begin{eqnarray} \label{31}
\Big\| \sum_{k=1}^{\infty} c_k v_k \Big\|_{L_p \cap L_{\infty}}
&=&
\max \left( \Big\| \sum_{k=1}^{\infty} c_k v_k \Big\|_{L_p}, \Big\| \sum_{k=1}^{\infty} c_k v_k \Big\|_{L_{\infty}} \right) \nonumber \\
&\asymp&
\max \left( \Big(\sum_{k=1}^{\infty} |c_k|^p \| v_k\|_{L_p}^p \Big)^{1/p}, \, \sup_{k \in \mathbb N} |c_k| \| v_k\|_{L_{\infty}} \right).
\end{eqnarray}
Firstly, let us assume that $1 \le p < 2$. If $\limsup_{k \rightarrow \infty} \| v_k\|_{L_p} > 0$, then selecting a further subsequence 
(and again keeping notation), we obtain the inequality
\begin{equation*}
\Big\| \sum_{k=1}^{\infty} c_k v_k \Big\|_{L_p \cap L_{\infty}} \geq c \, \Big(\sum_{k=1}^{\infty} |c_k|^p \Big)^{1/p},
\end{equation*}
which contradicts the right-hand estimate in (\ref{30}). So, $\lim_{k \rightarrow \infty} \| v_k\|_{L_p} = 0$, and then from (\ref{31}) 
for some subsequence of $\{v_k\}$ (we still keep notation) we have
\begin{equation} \label{32}
\Big\| \sum_{k=1}^{\infty} c_k v_k \Big\|_{L_p \cap L_{\infty}} \leq C_1 \, \sup_{k \in \mathbb N} |c_k|,
\end{equation}
and now the left-hand side of (\ref{30}) fails. Thus, if $1 \le p < 2$, inequality (\ref{30}) does not hold.

Let $p > 2$. Clearly, from (\ref{31}) it follows that
\begin{equation} \label{33}
\Big\| \sum_{k=1}^{\infty} c_k v_k\Big\|_{L_p \cap L_{\infty}} \leq C_2 \, \Big(\sum_{k=1}^{\infty} |c_k|^p \Big)^{1/p},
\end{equation}
and so the left-hand side estimate in (\ref{30}) cannot be true. Thus, (\ref{30}) fails for every $p \in [1, 2) \cup (2, \infty)$, and as a result 
we get
$$
\liminf_{k \rightarrow \infty} \| u_k \|_{L_p \cap L_{\infty}} > 0.
$$

Now, if $1 \le p < 2$, then, as above, $\lim_{k \rightarrow \infty}  \| v_k \|_{L_p} = 0$, and we come (for some subsequence of $\{v_k\}$) 
to inequality (\ref{32}). Clearly, 
\begin{equation} \label{34}
\Big\| \sum_{k=1}^{\infty} c_k u_k \Big\|_{L_p \cap L_{\infty}} \geq c \sup_{k \in \mathbb N} |c_k|,
\end{equation}
and from (\ref{32}) and (\ref{28}) it follows that for some $C > 0$ and all $(c_k) \in l_2$ we have
\begin{equation} \label{35}
C^{-1} \, \| (c_k) \|_{l_2} \leq \Big\| \sum_{k=1}^{\infty} c_k u_k \Big\|_{L_p \cap L_{\infty}} \leq C \| (c_k) \|_{l_2}.
\end{equation}
Therefore, the subspace $[u_k]$ is isomorphic in $L_p \cap L_{\infty}$ to $l_2$.

We show that the last claim holds also in the case $p > 2$. On the contrary, assume that the left-hand side of (\ref{35}) fails 
(note that the opposite side of (\ref{35}) follows from (\ref{28})). In other words, assume that there is a sequence 
$(c_k^n)_{k=1}^{\infty} \in l_2, n = 1, 2, \ldots$, such that $\|(c_k^n) \|_{l_2} = 1$ for all $n \in \mathbb N$ and
$$
\Big\| \sum_{k=1}^{\infty} c_k^n u_k \Big\|_{L_p \cap L_{\infty}} \rightarrow 0 ~ {\rm as} ~~ n \rightarrow \infty.
$$
Then, by (\ref{34}), we have $\sup_{k \in \mathbb N} |c_k^n| \rightarrow 0$ as $ n \rightarrow \infty$. Therefore, since
$$
\sum_{k=1}^{\infty} |c_k^n |^p \leq \big(\sup_{k \in \mathbb N} |c_k^n| \big)^{p-2} \sum_{k=1}^{\infty} |c_k^n |^2 = 
\big(\sup_{k \in \mathbb N} |c_k^n| \big)^{p-2},
$$
we have $\sum_{k=1}^{\infty} |c_k^n |^p \rightarrow 0$ as $ n \rightarrow \infty$. Combining this together with (\ref{33}), we 
obtain
$$
\Big\| \sum_{k=1}^{\infty} c_k^n v_k \Big\|_{L_p \cap L_{\infty}} \rightarrow 0 ~ {\rm as} ~~ n \rightarrow \infty,
$$
and so the left-hand estimate in (\ref{28}) does not hold. This contradiction shows that (\ref{35}) is valid for every 
$p \in [1, 2) \cup (2, \infty)$. Thus, the subspace $[u_k]$ is complemented in $L_p \cap L_{\infty}$ and isomorphic to $l_2$.
As an immediate consequence of that, we infer that $[u_k]$ is a complemented subspace of the space $L_p \cap L_{\infty}(E)$, 
where $E = \bigcup_{k=1}^{\infty} {\rm supp}\, u_k = \bigcup_{k=1}^{\infty} C_k$. Since by (\ref{29}) $m(E) \leq 1$, it follows that
$L_p \cap L_{\infty}(E)$ is isometric to $L_{\infty}(E)$. As a result we come to a contradiction, because $L_{\infty}$ does not contain 
any complemented reflexive subspace (cf. \cite[Theorem~5.6.5]{AK}).
\end{proof}

\begin{proof}[Proof of Theorem \ref{Thm1} for $p \in (1, 2) \cup (2, \infty)$]
Clearly, if $1<p<\infty$ then $L_p$ (and hence $L_p+L_\infty$) contains a complemented copy of $l_2$ (for instance, the span of the Rademacher sequence). 
Therefore, by applying Theorem~\ref{Thm4}, we complete the proof.
\end{proof}

Note that if $X$ is a symmetric space on $(0, \infty)$, then $X + L_{\infty}$ contains a complemented 
space isomorphic to $X[0, 1] = \{x \in X \colon {\rm supp}\, x \subset [0, 1] \}$ since 
$$
\| x \chi_{[0, 1]}\|_X \leq C_X \, \| x \|_{X+L_{\infty}} ~ {\rm for} ~ x \in X+L_{\infty},
$$
where $C_X \leq \max(2\, \| \chi_{[0, 1]}\|_X, 1)$. In fact, for $x \in X+L_{\infty}$, using estimate (4.2) from \cite[p. 91]{KPS}, we obtain
\begin{eqnarray*}
\| x \|_{X + L_{\infty}} 
&\geq&
\| x \chi_{[0, 1]} \|_{X + L_{\infty}} \geq \inf_{A \subset [0, 1]} (\| x \chi_{A}\|_X + \| x \chi_{[0, 1] \setminus A} \|_{L_{\infty}} ) \\
&\geq&
\inf_{A \subset [0, 1]} (\| x \chi_{A}\|_X + \frac{1}{2 \, \| \chi_{[0, 1]} \|_X} \| x \chi_{[0, 1] \setminus A} \|_X ) \geq \frac{1}{C_X} \| x \chi_{[0, 1]} \|_X.
\end{eqnarray*}
So, an inspection of the proofs of Theorems  \ref{Thm4} and \ref{Thm1} (in the case when $p \in (1, 2) \cup (2, \infty)$) shows 
that the following more general result is true. 

\begin{theorem} \label{Thm5} Suppose $X$ is a separable symmetric space on $(0, \infty)$ satisfying
either the upper $p$-estimate for $p > 2$ or lower $q$-estimate for $q < 2$.
Then the space $X \cap L_{\infty}$ does not contain any complemented subspace isomorphic to $l_2$.

If, in addition, the space $X[0, 1]$ contains a complemented subspace isomorphic to $l_2$, then
the spaces $X \cap L_{\infty}$ and $X + L_{\infty}$ are not isomorphic.
\end{theorem}

\section{ Concluding remarks related to the spaces $L_2 + L_{\infty}$ and $L_2 \cap L_{\infty}$}

We do not know whether the spaces $L_2 + L_{\infty}$ and $L_2 \cap L_{\infty}$ are isomorphic or not. 

\begin{problem}
Are the spaces $L_2 + L_{\infty}$ and $L_2 \cap L_{\infty}$ isomorphic? 
\end{problem}

We end up the paper with the following remarks related to the above problem.

\begin{remark} \label{Rem2} 
The predual spaces $L_1 \cap L_2$ and $L_1 + L_2$ for $L_2 + L_{\infty}$ and $L_2 \cap L_{\infty}$, respectively, 
are not isomorphic. 
\end{remark}

In fact, $L_1 \cap L_2$ is a separable dual space since $(L_2 + L_{\infty}^0)^* = L_2 \cap L_1$ (cf. \cite[Proposition 2(a)]{Di88}). 
Therefore, the space $L_1[0, 1]$ cannot be embedded in this space (cf. \cite[p. 147]{AK}) but 
$L_1 + L_2$ has a complemented subspace isomorphic to $L_1[0, 1]$, which completes our observation. 

\begin{remark} \label{Rem3} 
Either of the spaces $L_2 + L_{\infty}$ and $L_2 \cap L_{\infty}$ is isomorphic to a (uncomplemented) subspace of $l_\infty$, and hence $L_2 + L_{\infty}$ is isomorphic to a subspace of $L_2 \cap L_{\infty}$ and vice versa.
\end{remark}

To see this, for instance, for $L_2 + L_{\infty}$, it is sufficient to take arbitrary dense sequence of the unit ball of the space $L_1 \cap L_2$, say, $\{\varphi_n\}_{n=1}^\infty$, and to set
$$
Tx:=\Big(\int_0^\infty x(t)\varphi_n(t)\,dt\Big)_{n=1}^\infty\;\;\mbox{for all}\;\;x\in L_2 + L_{\infty}.$$ 
It is easy to see that this mapping defines an isometrical embedding of $L_2 + L_{\infty}$ into $l_\infty$.


\noindent
{\footnotesize Department of Mathematics\\
Samara National Research University, Moskovskoye shosse 34\\
443086, Samara, Russia}\\
{\it E-mail address:} {\tt astash56@mail.ru} \\

\noindent
{\footnotesize Department of Mathematics, Lule\r{a} University of Technology\\
SE-971 87 Lule\r{a}, Sweden}\\
{\it E-mail address:} {\tt lech@sm.luth.se} \\

\end{document}